\documentclass[11pt, a4]{amsart}
\usepackage{amsfonts, amstext, amsmath, amsthm, amscd, amssymb}
\usepackage{url}
\usepackage[all]{xypic}

\usepackage{multido}

\usepackage{enumerate}

\usepackage[all,graph]{xy}
\usepackage{psfrag}

\newcommand{\bp}{\begin{pmatrix}}
\newcommand{\ep}{\end{pmatrix}}
\newcommand{\be}{\begin{equation}}
\newcommand{\ee}{\end{equation}}
\newcommand{\ol}[1]{\overline{#1}}

\numberwithin{equation}{section}

\theoremstyle{plain}
\newtheorem{theorem}[equation]{Theorem}
\newtheorem{lemma}[equation]{Lemma}
\newtheorem{corollary}[equation]{Corollary}
\newtheorem{proposition}[equation]{Proposition}
\newtheorem*{claim*}{Claim}

\theoremstyle{definition}
\newtheorem{remark}[equation]{Remark}

\newtheorem{definition}[equation]{Definition}

\def\Z{\mathbb Z}

\def\Q{\mathbb Q}

\def\sm{\setminus}

\def\toiso{\xrightarrow{\cong}}

\def\zt{\Z[t,t^{-1}]}

\def\bp{\begin{pmatrix}}
\def\ep{\end{pmatrix}}
\def\ba{\begin{array}}
\def\ea{\end{array}}
\def\bn{\begin{enumerate}}
\def\en{\end{enumerate}}

\def\op{\operatorname}

\def\PD{\op{PD}}

\DeclareMathOperator\inc{inc}

\DeclareMathOperator\Hom{Hom}

\DeclareMathOperator\Id{Id}

\DeclareMathOperator\Bl{Bl}

\def\wti{\widetilde}

\begin{document}

\title{Twisted Blanchfield pairings and decompositions of 3-manifolds}

\author{Stefan Friedl}
\address{Fakult\"at f\"ur Mathematik\\ Universit\"at Regensburg\\   Germany}
\email{sfriedl@gmail.com}

\author{Constance Leidy}
\address{Department of Mathematics\\ Wesleyan University\\ Wesleyan Station, Middletown, CT 06459, USA}
\email{cleidy@wesleyan.edu}

\author{Matthias Nagel}
\address{
D\'epartement de Math\'ematiques, Universit\'e du Qu\'ebec \`a Montr\'eal, QC, Canada}
\email{nagel@cirget.ca}

\author{Mark Powell}
\address{
D\'epartement de Math\'ematiques, Universit\'e du Qu\'ebec \`a Montr\'eal, QC, Canada}
\email{mark@cirget.ca}


\def\subjclassname{\textup{2010} Mathematics Subject Classification}
\expandafter\let\csname subjclassname@1991\endcsname=\subjclassname
\expandafter\let\csname subjclassname@2000\endcsname=\subjclassname

\begin{abstract}
We prove a decomposition formula for twisted Blanchfield pairings of 3-manifolds.
As an application we show that the twisted Blanchfield pairing of a 3-manifold obtained from a $3$-manifold $Y$ with a representation $\phi \colon \Z[\pi_1(Y)] \to R$, infected by a knot $J$ along a curve $\eta$ with $\phi(\eta) \neq 1$, splits orthogonally as the sum of the twisted Blanchfield pairing of $Y$ and the ordinary Blanchfield pairing of the knot~$J$, with the latter tensored up from~$\zt$ to~$R$.
\end{abstract}
\maketitle

\section{Introduction}
Given an oriented knot $P$ in $S^3$, together with an oriented unknot~$\eta$ in its
complement, and another oriented knot~$C$ in $S^3$, we can form
the $3$-manifold $S^3 \sm \nu \eta \cup S^3 \sm \nu C$, identifying
the two boundary tori by mapping the meridian of each knot to the longitude of the
other.
Take the image of the knot $P \subset S^3 \sm \nu \eta \cup S^3 \sm \nu C$
under the diffeomorphism of this manifold to $S^3$, to obtain the \emph{satellite knot} $P(J, \eta)$,
 with \emph{pattern}~$P$, \emph{companion}~$C$ and \emph{infection curve}~$\eta$.
Seifert~\cite{Seifert50} proved the elegant formula
\[ \Delta_{P(C,\eta)}(t) = \Delta_P(t) \cdot \Delta_C(t^\omega) \]
for the Alexander polynomial~$\Delta_{P(C,\eta)}(t)$, expressing it in terms
of the Alexander polynomials of the pattern and the companion and the winding number $\omega = \ell k(P,\eta)$.
This was extended to the Blanchfield form by Livingston and Melvin~\cite[Theorem 2]{LM86} as
\[ \Bl_{P(C, \eta)}(t) = \Bl_P(t) \oplus \Bl_C(t^\omega). \]
For twisted Alexander polynomials, a similar infection formula
was given by Kirk and Livingston~\cite[Theorem 3.7]{KL}.
In this article we obtain such a formula for twisted Blanchfield pairings.
Moreover, we generalise from satellites operation on knots to infections
of $3$-manifolds by knot complements, and in fact even further
to two $3$-manifolds glued together along a boundary torus.

Here are some definitions and conventions. In this paper a ring $R$ is
always equipped  with  (a possibly trivial) involution. For example, we view any
group ring $\Z[\pi]$ as a ring with involution in the canonical way.
Furthermore, all ring homomorphisms will be involution preserving, that is
morphisms of rings-with-involution. Given a left $R$-module $M$, we denote the
right $R$-module defined using the involution on $R$ by $\overline{M}$.

Now let $R$ be an Ore domain with (possibly trivial) involution. Let $Q$ be the Ore
localisation of $R$, i.e.\ the (skew) field of fractions of $R$, which inherits an involution from $R$. We refer to
\cite{Pa77}, \cite{St75} for details on Ore domains and the Ore localisation.  The Ore condition guarantees that every left fraction is also a right fraction, so that the field of fractions can be defined.

A \emph{linking pairing} on a torsion left $R$-module $M$ is a morphism $B_M \colon M \to M^{\wedge}:= \ol{\Hom_{R}(M,Q/R)}$ of left $R$-modules.
The map $B_M$ is the adjoint of a sesquilinear pairing $B_M \colon M \times M \to Q/R$, and henceforth we identify the two notions without comment.
A linking pairing $B_M$ is said to be \emph{nonsingular} if $B_M$ is an isomorphism, and \emph{hermitian} if $B_M = B_M^{\wedge}$.  A linking pairing that satisfies both of these properties is called a \emph{linking form}.
A \emph{morphism} of linking pairings $\psi \colon (M,B_M) \to (N,B_N)$ is an $R$-module homomorphism $\psi \colon M \to N$ for which
\[B_M = \psi^*B_N := \psi^{\wedge} \circ B_N \circ \psi.\]
 An \emph{isomorphism} of linking pairings is defined to be a morphism of linking pairings $\varphi \colon (M,B_M) \to (N,B_N)$ for which $\varphi \colon M \toiso N$ is an isomorphism.

Now let $Y$ be a  3-manifold with empty or toroidal boundary.
Here and throughout the paper we assume that all 3-manifolds are compact, oriented and connected. Let $\phi \colon
\Z[\pi_1(Y)] \to R$ be a morphism to the Ore domain $R$ such that $H_*(Y;Q)=0$.
Under this hypothesis the, \emph{twisted Blanchfield pairing} of $(Y,\phi)$,
\[\Bl_{Y,\phi}\colon H_1(Y;R)\times H_1(Y;R)\to Q/R,\]
is defined.
We will recall the definition in detail in Section~\ref{section:defn-of-TBF}. If $Y=S^3\setminus \nu J$ is the exterior of an oriented knot $J$ (here $\nu J$  denotes an open tubular neighbourhood around $J$) and the morphism~$\phi\colon \Z[\pi_1(S^3\setminus \nu J)]\to \Z[t,t^{-1}]$ is induced by the abelianisation map, then
\[ H_1(S^3\sm \nu J;\zt)\times H_1(S^3\setminus \nu J;\zt)\to \Q(t)/\Z[t,t^{-1}]\]
is precisely the classical Blanchfield pairing $\Bl_J$ on the Alexander module of the knot; see \cite{Bl57}. The more general twisted Blanchfield pairings (sometimes referred to as higher-order Blanchfield pairings) first appeared in the seminal work of Cochran-Orr-Teichner \cite[Theorem~2.13]{COT03}.
The theory of twisted Blanchfield pairings was further developed by Leidy~\cite{Le06} and played a major r\^ole in the work of Cochran-Harvey-Leidy \cite{CHL08,CHL09,CHL11} and in \cite{Fr13, Bu14, Ch14, Ja17}.

In general it is difficult to give a useful description of twisted Blanchfield pairings over a non-commutative ring $R$.
The next theorem gives a decomposition formula for Blanchfield pairings, and thus allows the computation of Blanchfield pairings to be broken up into hopefully easier pieces.

\begin{theorem}[Orthogonal decomposition theorem]\label{thm:decomp-of-3-mflds-Bl-intro}
Let $Y$ be a 3-manifold with empty or toroidal boundary and let $Y=A\cup_T B$ be a decomposition of $Y$ along a torus $T$ into two 3-manifolds $A$ and $B$.
Let $R$ be an Ore domain with involution and let $\phi\colon \Z[\pi_1(Y)]\to R$ be a
 morphism such that $H_*(T;Q)=0$ and such that $H_*(Y;Q)=0$.
Then $H_*(A;Q)=0$, $H_*(B;Q)=0$, and the inclusion maps
$i_A \colon A\to Y$ and $i_B \colon B\to Y$ induce a morphism of linking pairings
\[ i_A + i_B \colon (H_1(A;R)\oplus H_1(B;R), \Bl_{A, \phi|_A} \oplus \Bl_{B,\phi|_B}) \to (H_1(Y;R), \Bl_{Y,\phi}).\]
\end{theorem}

This theorem can be used for many different purposes. For example it can be used to prove a formula relating the Blanchfield form of a connected sum of knots $K\# J$ to the Blanchfield forms of the knots $K$ and $J$.
Arguably the most important application of Theorem~\ref{thm:orthogonal-decomp-of-Bl-infected-intro} is to infection of a 3-manifold by a knot, as in the aforementioned papers by Cochran-Harvey-Leidy, Burke, Cha, Franklin and Jang.

\begin{definition}\label{defn:infection}
Let $Y$ be a $3$-manifold and let $\eta \subset Y$ be an oriented embedded circle.   Denote the exterior by $Y(\eta) := Y \sm \nu \eta$.
Furthermore, let $J \subset S^3$ be an oriented knot with exterior~$E_J:=S^3 \sm \nu J$.
An \emph{infection} of $Y$ by $J$ is the $3$-manifold
\[Y_J := Y(\eta) \cup E_J\]
where the meridian of $\eta$ is glued to the zero-framed longitude of $J$, and some longitude of $\eta$ is glued to the meridian of $J$.
\end{definition}

There is some indeterminacy in the choice of the longitude of $\eta$, and changing the isotopy class of $\eta$ can change $Y_J$. However we will see that the twisted Blanchfield pairing only depends on the homotopy class of $\eta$.
There exists a degree one map $E_J \to E_U=S^1\times D^2$, which restricted to the boundary is a diffeomorphism
that preserves the meridian and longitude. This map extended by the identity defines a degree one map $f \colon Y_J \to Y$.

In order to state the next result we need to introduce more notation.
Let $\pi$ be a group and let $\eta\in \pi$. Given an Ore domain $R$ and a
morphism $\phi\colon \Z[\pi]\to R$, we say that $\phi$ is \emph{$\eta$-regular} if the induced map $\Z[\langle \eta\rangle]\to R$ is a monomorphism.
If $J\subset S^3$ is an oriented knot and $\phi\colon \Z[t,t^{-1}]=\Z[\langle t\rangle]\to R$ is a $t$-regular homomorphism, then we can consider the tensor product $R \otimes_{\Z[t,t^{-1}]} \Bl_{J}$. More precisely, we have the pairing
\[ \begin{array}{rcl} R\otimes_{\Z[t^{\pm 1}]} H_1(E_J;\Z[t^{\pm 1}])\times R\otimes_{\Z[t^{\pm 1}]} H_1(E_J;\Z[t^{\pm 1}])&\to & Q/R\\
((r\otimes h),(r'\otimes h'))&\mapsto & r\phi(\Bl_J(h,h'))\ol{r}'.\end{array}\]
In the case of an infection, $E_J \subset Y_J$, and the restriction of $\phi \circ f_* \colon \Z[Y_J] \to R$ to $\Z[\pi_1(E_J)] \to R$ factors through $\Z[\langle \eta \rangle] = \zt$.  If $\phi$ is $\eta$-regular, we can identify $$H_1(E_J;R) \cong R \otimes_{\zt} H_1(E_J;\zt).$$  In Lemma~\ref{lemma:tensoring-up} we also prove that under this identification $\Bl_{E_J} = R \otimes_{\Z[t,t^{-1}]} \Bl_{J}$.
The following theorem is the second main result of this paper.

\begin{theorem}[Infection by a knot for twisted Blanchfield pairings]\label{thm:orthogonal-decomp-of-Bl-infected-intro}
Let $Y$ be a $3$-manifold with empty or toroidal boundary, let $\eta \subset Y$ be a simple closed curve and let $J\subset S^3$ be an oriented knot.
Furthermore, let $\phi\colon \Z[\pi_1(Y)]\to R$ be an
$\eta$-regular homomorphism to an Ore domain such that
$H_*(Y;Q)=0$.
Then there is an isomorphism $\psi \colon H_1(Y;R) \oplus H_1(E_J;R) \toiso H_1(Y_J;R)$
(defined in Corollary~\ref{cor:theIsometryMap}), that induces an isomorphism
\[ \Bl_{Y,\phi} \oplus (R \otimes_{\zt} \Bl_{J})  \toiso  \Bl_{Y_J,\phi \circ f_*}\]
of linking pairings.
\end{theorem}

\begin{remark}
The statement of Theorem~\ref{thm:orthogonal-decomp-of-Bl-infected-intro} is a  generalisation of \cite[Theorem~4.6]{Le06}. The proof of~\cite[Theorem~4.6]{Le06} is problematic, since in the second diagram on page~765, the square involving Poincar\'{e} duality does not commute in general.
\end{remark}

The paper is organised as follows. In Section~\ref{section:defn-of-TBF} we give
the definition of twisted Blanchfield pairings.
Section~\ref{section:proof-of-orthog-decomp-thm} gives the proof of the
orthogonal decomposition Theorem~\ref{thm:decomp-of-3-mflds-Bl-intro}, our main
technical result. Then in Section~\ref{section:Bl-pairing-infection}, we apply
Theorem~\ref{thm:decomp-of-3-mflds-Bl-intro} to prove
Theorem~\ref{thm:orthogonal-decomp-of-Bl-infected-intro}.

\subsection*{Acknowledgements}
We thank the referee for helpful feedback.
SF, MN and MP  gratefully acknowledge the support provided by the SFB 1085 `Higher
Invariants' at the University of Regensburg, funded by the Deutsche
For\-schungsgemeinschaft (DFG).
This paper was written while MP was a visitor at the Max Planck Institute for Mathematics in Bonn.

\section{Twisted Blanchfield pairings}\label{section:defn-of-TBF}
Let $X$ be connected $CW$-complex and let $Y\subset X$ be a possibly empty subcomplex. Furthermore, let $R$ be a ring and let $M$ be a $(R,\Z[\pi_1(X)])$-bimodule. We can define the cellular chain complex
\[ M\otimes_{\Z[\pi]} C_*(\wti{X},\wti{Y};\Z),\]
where $\wti{X}$ is the universal cover of $X$,
and its dual chain complex
\[ \Hom_{\text{right-}\Z[\pi]}\Big(\ol{C_*(\wti{X},\wti{Y};\Z)},M\Big).\]
Here $\wti{Y}$ is the pullback covering space of $\wti{X} \to X$ under the inclusion $Y \subset X$.
Both chain complexes are naturally chain complexes of left $R$-modules.
We denote the corresponding homology groups as $H_*(X,Y;M)$ and $H^*(X,Y;M)$, which are again left $R$-modules.

 Let $R$ be an Ore domain with involution and let $\phi \colon \Z[\pi_1(X)] \to R$ be a morphism. This allows us to view $R$, $Q$ and $Q/R$ as $\Z[\pi_1(X)]$-right modules.
Using the fact that $\phi$ is a morphism of rings with involution, it is straightforward to verify that
\[ \ba{rcl}
\Hom_{\text{right-}\Z[\pi]}\Big(\ol{C_*(\wti{X},\wti{Y};\Z)},Q/R\Big)&\to&\ol{ {\op{Hom}_{\text{left-}R}(R\otimes_{\Z[\pi]} C_*(\wti{X},\wti{Y};\Z),{Q/R})}}\\
f&\mapsto & \left( r\otimes \sigma\mapsto r \cdot \ol{ f(\sigma)}\right)
\ea \]
is a well-defined isomorphism of chain complexes of left $R$-modules. For the left $R$-action on the domain, we use the involution on $Q/R$ to convert it to a right $R$-module.  The isomorphism of chain complexes above induces a homomorphism
\[\kappa \colon H^i(X,Y;Q/R)\to \ol{\op{Hom}_{\text{left-}R}(H_i(X,Y;R),Q/R)}\]
of left $R$-modules.

Now let $N$ be a 3-manifold with empty or toroidal boundary.
Throughout the remainder of this section we assume that $H_*(N;Q)=0$.
We consider the following  sequence of homomorphisms
\[\xymatrix @C-0.2cm@R0.8cm{H_1(N;R) \ar[d]_{\Bl_{N,\phi}}\ar[r]^-{\PD}& H^2(N,\partial N;R)\ar[r]^-{\beta^{-1}}&  H^1(N,\partial N;Q/R) \ar[d]^-{\kappa}\\
\ol{\Hom_R(H_1(N;R),Q/R)} &&\ar[ll]^{i^*} \ol{\Hom_R(H_1(N,\partial N;R),Q/R)}.}\]
Here:
\begin{enumerate}
\item $\PD\colon H_1(N;R)\to H^2(N,\partial N;R)$ denotes the Poincar\'{e}-Lefschetz duality map \cite[Section 5.2.2]{DK};
\item $\beta \colon H^1(N,\partial N;Q/R)\to  H^2(N,\partial N;R)$ denotes the Bockstein homomorphism, which is an isomorphism since our assumption that $H_*(N;Q)=0$ implies by Poincar\'e duality  that $H^*(N,\partial N;Q)=0$;
\item $\kappa\colon H^1(N,\partial N;Q/R) \to \Hom_R(H_1(N,\partial N;R),Q/R)$ denotes the Kronecker evaluation map defined above;
\item $i\colon H_1(N;R)\to H_1(N,\partial N;R)$ denotes the map from the long exact sequence of the pair.
\end{enumerate}
We refer to the pairing induced by the composition of these four maps,
\[ \ba{rcl} \Bl_{N,\phi}\colon H_1(N;R)\times H_1(N;R)&\to & Q/R \\
(a,b)&\mapsto & \Bl_{N,\phi}(b)(a),\ea\]
as the Blanchfield pairing of $(N,\phi)$. By definition it is sesquilinear, meaning that it is linear in the first entry and conjugate-linear in the second entry.
In favourable situations the Blanchfield pairing can also be shown to be hermitian and nonsingular, but we do not investigate these properties in this article.

\section{Proof of the orthogonal decomposition theorem}\label{section:proof-of-orthog-decomp-thm}

If $Y$ is a $3$-manifold, $\phi\colon \Z[\pi_1(Y)]\to R$ is a morphism, and $X\subset Y$ is a connected submanifold, then let us also denote the restriction of $\phi$ to $\Z[\pi_1(X)]$ by $\phi$.  For the convenience of the reader, we recall the statement of Theorem~\ref{thm:decomp-of-3-mflds-Bl-intro} from the introduction.

\begin{theorem}\label{thm:decomp-of-3-mflds-Bl}
Let $Y$ be a 3-manifold with empty or toroidal boundary and let $Y=A\cup_T B$ be a decomposition of $Y$, along a torus $T$, into two  3-manifolds $A$ and $B$.
Let $R$ be an Ore domain with involution and let $\phi\colon \Z[\pi_1(Y)]\to R$ be a morphism such that $H_*(T;Q)=0 = H_*(Y;Q)$.
Then $H_*(A;Q)=0$ and $H_*(B;Q)=0$ and the inclusion maps
$i_A \colon A\to Y$ and $i_B \colon B\to Y$ induce a morphism of linking pairings
\[  i_A + i_B \colon (H_1(A;R)\oplus H_1(B;R), \Bl_{A, \phi} \oplus \Bl_{B,\phi}) \to (H_1(Y;R), \Bl_{Y,\phi}).\]
\end{theorem}

\begin{proof}
In an attempt to keep the notation at a reasonable level we make the extra assumption that $Y$ is closed.
The proof we provide also goes through without problems in the case that $Y$ has boundary.

The Mayer-Vietoris sequence for $Y=A\cup_T B$ with $Q$-coefficients, and our hypothesis that $H_*(T;Q)=0$, implies that  $H_*(A;Q)=0$ and $H_*( B;Q)=0$. In particular the Blanchfield pairings on $A$ and $B$ are defined.
Recall that, given an $R$-module $P$, we denote $\ol{\Hom_R(P,Q/R)}$ by $P^{\wedge}$.  Consider the following diagram
\begin{equation*}
   \xymatrix @C1.2cm @R0.7cm { H_1(A;R)\ar[r]\ar[d]^{\Bl_{A}} &  H_1(Y;R)\ar[d]^{\Bl_{Y}}& \ar[d]^{\Bl_{B}} \ar[l] H_1(B;R)    \\
H_1(A;R)^\wedge& H_1(Y;R)^{\wedge} \ar[l]\ar[r] &H_1(B;R)^\wedge},
\end{equation*}
where the horizontal maps are induced by inclusion. Here and throughout the proof we omit the $\phi$ from the notation for the Blanchfield pairing.
We make the following observations.
\begin{enumerate}[(I)]
\item\label{item:proof-aspect-I} The statement that $(H_1(A;R), \Bl_{A}) \to (H_1(Y;R),\Bl_{Y})$ is a morphism of linking pairings is equivalent to the statement that the left square commutes.
\item\label{item:proof-aspect-II} The statement that $(H_1(B;R), \Bl_{B}) \to (H_1(Y;R),\Bl_{Y})$ is a morphism of linking pairings is equivalent to the statement that the right square commutes.
\item\label{item:proof-aspect-III} The statement that the images of $H_1(A;R)$ and $H_1(B;R)$ are orthogonal is equivalent to the statement that the map $H_1(A;R)\to H_1(B;R)^\wedge$, from the top left to the bottom right, and
also the map $H_1(B;R)\to H_1(A;R)^\wedge$, from the top right to the bottom left, are both the zero map.
\end{enumerate}
Now consider the following diagram.
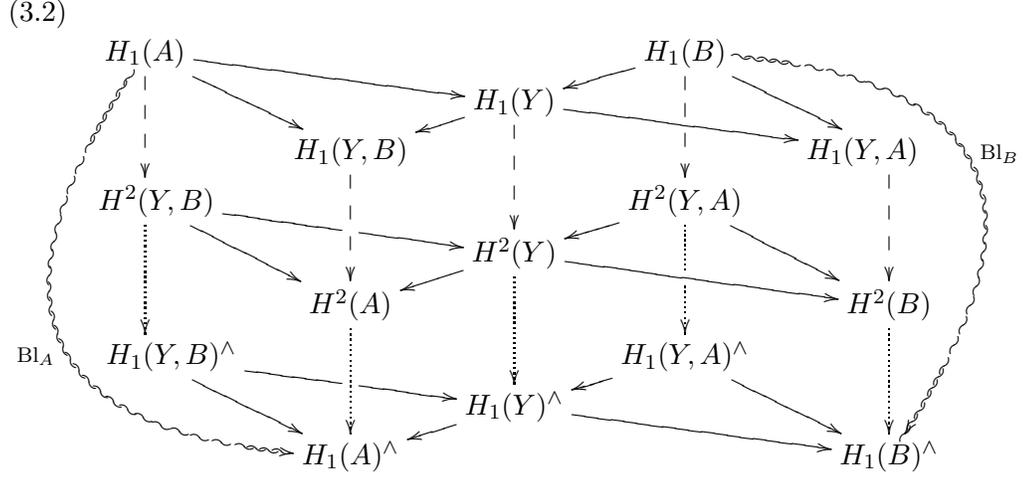
\begin{equation}\label{equation:big-diagram}
\xymatrix @C-0.31cm @R-0.80cm{ H_1(A)\ar@{~>}@/_6.5pc/[ddddddddr]_{\Bl_A}\ar[drr]\ar[ddr] \ar@{-->}[ddd] &&&H_1(B)\ar@{~>}@/^6.5pc/[ddddddddr]^{\Bl_B}\ar[dl]\ar[ddr]
\ar@{-->}[ddd]|!{[dl];[ddr]}\hole\\
&&H_1(Y)\ar[dl]\ar[drr]\ar@{-->}[ddd]\\
&H_1(Y,B)\ar@{-->}[ddd]&&&H_1(Y,A)\hspace{0.7cm}\ar@{-->}[ddd]\\
\hspace{0.3cm} H^2(Y,B)\ar[drr]|!{[ur];[ddr]}\hole\ar@{.>}[ddd]\ar[ddr] &&&H^2(Y,A)\ar[dl]\ar[ddr] \ar@{.>}[ddd]|!{[dl];[ddr]}\hole \\
&&H^2(Y)\ar[dl]\ar[drr]\ar@{.>}[ddd]\\
&H ^2(A)\ar@{.>}[ddd]&&&H^2(B)\ar@{.>}[ddd]\\
\hspace{0.7cm} H_1(Y,B)^\wedge\ar[ddr] \ar[drr]|!{[dr];[ur]}\hole &&& H_1(Y,A)^\wedge  \ar[dl]\ar[ddr]\\
&&H_1(Y)^\wedge  \ar[drr]\ar[dl]\\
&H_1(A)^\wedge &&&H_1(B)^\wedge\\
}
\end{equation}
For space reasons we omit the $R$-coefficients.
Before we discuss the maps in the diagram and the commutativity, we give a quick guide to the diagram.
The diagram (\ref{equation:big-diagram}) consists of four parts.

\begin{enumerate}[(a)]
\item\label{item:a} The large parallelogram spanned by the groups~$H_1(A;R)$, $H_1(Y,A;R)$, $H_1(B;R)^\wedge$ and $H_1(Y,B;R)^\wedge$, comprising four smaller parallelograms.  The large parallelogram contains three rows, each of which is a portion of the long exact sequence corresponding either to the pair  $(Y,A)$ or the pair $(Y,B)$.
\item\label{item-b} The parallelogram spanned by the groups~$H_1(B;R)$, $H_1(Y,B;R)$, $H_1(A;R)^\wedge$ and $H_1(Y,A;R)^\wedge$ is defined in the same way as the previous parallelogram from (\ref{item:a}), except that we swapped the r\^oles of $A$ and $B$.
\item There are diagonal maps towards the left and the right of the diagram that connect the two large parallelograms described in (\ref{item:a}) and (\ref{item-b}), e.g.\ the map $H_1(A) \to H_1(Y,B)$ on the left and the map $H_1(B) \to H_1(Y,A)$ on the right.
\item The undulating arrows are given by the maps $\Bl_A$ and $\Bl_B$ defining the Blanchfield pairings on $A$ and $B$.
\end{enumerate}

The above discussion shows, in particular, that the composition of two collinear solid maps is zero.
In the hope of facilitating comprehension, we use four different ways of depicting maps:

\begin{enumerate}[(1)]
\item The solid maps are all inclusion induced maps. For example, the map $H_1(B;R)\to H_1(Y,A;R)$ on the top right is induced by the inclusion of the pair $(B,\emptyset)$ to $(Y,A)$.
\item The dashed arrows going down are the inverses of the maps given by capping with the fundamental class $[Y]$.
\item The dotted arrows going down are the maps given by $\kappa\circ \beta^{-1}$.
\item As mentioned above, the undulating arrows are given by the maps $\Bl_A$ and $\Bl_B$ defining the Blanchfield pairings on $H_1(A;R)$ and $H_1(B;R)$.
\end{enumerate}
Now we argue that the above diagram is commutative. More precisely, we show that each parallelogram and each triangle commutes.
\begin{enumerate}[(i)]
\item All the triangles involve  only  inclusion induced maps, hence they commute.
\item The Bockstein homomorphism and the Kronecker evaluation map are functorial. In particular they commute with inclusion induced maps. It follows that all parallelograms involving the dotted arrows commute.
\item\label{item:parallelogram-3} Consider the parallelogram spanned by $H_1(A;R)$, $H_1(Y,A;R)$, $H^2(B;R)$ and $H^2(Y,B;R)$, comprising two smaller parallelograms. This piece of the diagram commutes by  \cite[Corollary~VI.8.6]{Br93}.
\item Similarly to (\ref{item:parallelogram-3}), the parallelogram spanned by $H_1(B;R)$, $H_1(Y,B;R)$, $H^2(A;R)$ and $H^2(Y,A;R)$ also commutes.
\item\label{item:piece-of-diagram-5} Next we consider the following piece of the above diagram
\[
\xymatrix @C-0.11cm @R-0.980cm{ H_1(A;R)\ar[drr]\ar[ddr] \ar@{-->}[ddd] \\
&&H_1(Y;R)\ar[dl]\ar@{-->}[ddd]\\
&H_1(Y,B;R)\ar@{-->}[ddd]&\\
\hspace{0.3cm} H^2(Y,B;R)\ar[drr]|!{[ur];[ddr]}\hole\ar[ddr] & \\
&&H^2(Y;R)\ar[dl]\\
&H ^2(A;R)&
}
\]
We have already shown that the top and bottom triangle commute, that the parallelogram in the back commutes and that the parallelogram on the right commutes. It is then straightforward to verify that the parallelogram on the left also commutes.
\item The same argument as in (\ref{item:piece-of-diagram-5}) shows that the parallelogram given by $H_1(B;R)$, $H_1(Y,A;R)$, $H^2(Y,A;R)$ and $H^2(B;R)$ commutes.
\item It remains to show that the pieces of the diagram involving the undulating arrows commute. By symmetry it suffices to show that the piece involving $\Bl_A$ commutes. Going back to the definition of $\Bl_A$, we have to show that the following diagram commutes.
\begin{equation}\label{equation:commuting-diagram-2}
\xymatrix@C1.6cm@R1.1cm{
H_1(A)\ar@{-->}[r]_-{\cong} \ar[d]_-{=}^-{\theta} & H^2(Y, B)\ar@{-->}@/_1pc/[l]_{\cap [Y]}\ar@{.>}[r]^-{\kappa \circ \beta^{-1}} \ar[d]^-{\Upsilon} &H_1(Y, B)^\wedge\ar[r]\ar[d] & H_1(A)^\wedge \ar[d]
\\
 H_1(A)\ar@{~>}@/_2pc/[rrr]_{\Bl_A}\ar@{-->}[r]_-\cong &\ar@{-->}@/_1pc/[l]_{\cap [A]}H^2(A,\partial A)\ar@{.>}[r]^-{\kappa \circ \beta^{-1}}&H_1(A,\partial A)^\wedge\ar[r] & H_1(A)^\wedge
}\end{equation}
Here $\Upsilon$ denotes the inclusion induced map. In the diagram, we once more suppress the $R$-coefficients from the notation. Furthermore we use the same conventions for the arrows as above. Therefore we see, with the same arguments as above, that the central square and the square on the right commute.

To see that the square on the left commutes, we return to the definition of the cap product that defines the Poincar\'{e} duality isomorphisms.  Let $[A]$ and $[B]$ be fundamental classes in $H_3(A,\partial A;\Z)$ and $H_3(B,\partial B;\Z)$ respectively.  Then $[Y]:=i_A([A])+i_B([B])$ is a fundamental class for $Y$.  Denote Eilenberg-Zilber diagonal chain approximation maps, for example as arising from the Alexander-Whitney diagonal approximation map~\cite[Chapter~VI]{Br93}, by $\Delta$.  Then $$\Delta([Y]) = \Delta([A]) + \Delta([B]) \in \ol{C_*(Y;R)} \otimes_{R} C_*(Y;R).$$  Let $[f] \in H^2(Y,B)$ and we denote the dual of the excision isomorphism
by $e^* \colon H^2(Y,B) \to H^2(A,\partial A)$.
The cochain $f \colon C_2(Y,B;R) \to R$ is a function which vanishes on chains of $B$.  This explains the penultimate equality of the following (co-)chain level computation.  We have
\begin{align*}
  f \cap [Y] &= (f \otimes \Id)\Delta([Y]) \\
  &= (f \otimes \Id)\Delta([A] + [B]) \\
 &= (f \otimes \Id)(\Delta([A]) + \Delta([B])) \\
&= (f \otimes \Id)(\Delta([A])) + (f \otimes \Id)(\Delta([B])) \\
&= (f \otimes \Id)(\Delta([A])) + 0 \\
&= e^*(f) \cap [A].
\end{align*}
Thus the left square of (\ref{equation:commuting-diagram-2}) commutes with the arrows pointing to the left. Denote capping with $[Y]$, $[A]$ by $\rho_Y$, $\rho_A$ respectively.  Then we have shown that $\theta \circ \rho_Y = \rho_A \circ \Upsilon$.  However the horizontal dashed maps are isomorphisms, so it follows that the left square commutes with both directions of the arrows.
For:
\[\rho_A^{-1}\circ\theta = \rho_A^{-1}\circ\theta \circ \rho_Y \circ \rho_Y^{-1} = \rho_A^{-1}\circ\rho_A \circ \Upsilon \circ \rho_Y^{-1} = \Upsilon \circ \rho_Y^{-1},\]
as desired.
\end{enumerate}

This concludes the proof that the big diagram~(\ref{equation:big-diagram}) commutes. But it is now straightforward to deduce from the big diagram that the original statements contained in~(\ref{item:proof-aspect-I}), (\ref{item:proof-aspect-II}) and (\ref{item:proof-aspect-III}) hold.
\end{proof}

\section{The Blanchfield pairing of an infection}\label{section:Bl-pairing-infection}

For the convenience of the reader, we recall the statement of Theorem~\ref{thm:orthogonal-decomp-of-Bl-infected-intro} from the introduction with a little more detail.

\begin{theorem}\label{thm:orthogonal-decomp-of-Bl-infected}
Let $Y$ be a $3$-manifold with empty or toroidal boundary, let $\eta \subset Y$ be a simple closed curve and let $J\subset S^3$ be an oriented knot.
Let $\phi\colon \Z[\pi_1(Y)]\to R$ be an
$\eta$-regular morphism to an Ore domain with involution such that
$H_*(Y;Q)=0$.
Then the map $\psi$ defined in Corollary~\ref{cor:theIsometryMap} is an isomorphism $H_1(Y;R) \oplus H_1(E_J;R) \toiso H_1(Y_J;R)$.
There is an identification $R \otimes_{\zt} \Bl_J \toiso \Bl_{E_J}$.
Using this, $\psi$ induces an isomorphism of linking pairings
\begin{align*}
\psi \colon & \big(H_1(Y;R) \oplus (R \otimes_{\zt} H_1(E_J;\zt)), \Bl_{Y,\phi} \oplus (R \otimes_{\zt} \Bl_{J})\big) \\  \toiso & \big(H_1(Y_J;R), \Bl_{Y_J,\phi \circ f_*}\big).
\end{align*}
\end{theorem}

\subsection{The homology of the infected 3-manifold}\label{section:homology-infection}
Let $Y$ be a $3$-manifold with empty or toroidal boundary, let $\eta \subset Y$ be a simple closed curve and let $J\subset S^3$ be an oriented knot.
Furthermore, let $\phi\colon \Z[\pi_1(Y)]\to R$ be an
$\eta$-regular morphism to an Ore domain with involution.

The restriction to $\pi_1(E_J) \to R$ factors through the abelianisation map $\pi_1(E_J)\to H_1(E_J;\Z)\toiso \Z$, which implies $H_1(E_J;R) \cong R \otimes_{\Z[\Z]} H_1(E_J;\Z[\Z])$.

\begin{lemma}\label{lemma:homology-of-Y}~
We write $T=\partial \overline{\nu \eta}=\partial E_U=\partial E_J$.
 The inclusion maps $Y(\eta)\to Y_J$, $E_J\to Y_J$ and $Y(\eta)\to Y$, and the degree one map $f\colon Y_J\to Y$, induce a commutative diagram
  \[\xymatrix@R0.7cm{H_1(Y(\eta);R)/H_1(T;R)\oplus H_1(E_J;R) \ar[d]^-{(\Id\oplus 0)} \ar[r]^-\cong & H_1(Y_J;R)  \ar[d]^-{f_*}  \\
 H_1( Y(\eta);R)/H_1(T;R) \ar[r]^-\cong & H_1(Y;R),}\]
for which the horizontal maps are isomorphisms.
\end{lemma}

\begin{proof}
Below, we consider the Mayer-Vietoris sequences for
\[ Y=Y(\eta) \cup_T \overline{\nu \eta} = Y(\eta) \cup_T E_U \mbox{ and }Y_J = Y(\eta) \cup_T E_J.\]
 Note that $H_1(\nu \eta;R)=H_1(E_U;R)=0$ since $\phi(\eta) \neq 1$.  We also observe that the map $H_1(T;R)=H_1(\partial E_J;R) \to H_1(E_J;R)$ is the zero map, since this map is given by tensoring up $H_1(\partial E_J;\zt) \to H_1(E_J;\zt)$, but the latter map is the zero map.
 To see this, first note that $H_1(\partial E_J;\zt) \cong H_1(S^1 \times \mathbb{R};\Z) \cong \Z$, generated by the zero-framed longitude of $J$.  Then $H_1(E_J;\zt) \cong \pi_1(E_J)^{(1)}/\pi_1(E_J)^{(2)}$, and the longitude of any knot is a commutator of curves on a Seifert surface. As curves on the Seifert surface are commutators, the longitude is a double commutator in $\pi_1(E_J)^{(2)}$ and therefore vanishes in $H_1(E_J;\zt)$.
  The map $H_0(\partial Y(\eta);R) \to H_0(E_L;R)$ is an isomorphism for $L=J$ and $L=U$.  Note that $Y_U=Y$.

The computations above show that the Mayer-Vietoris sequences give rise to the
 following commutative diagram with exact rows:
\[\xymatrix@R0.7cm{H_1(T;R) \ar[r]^-{(\inc,0)} \ar[d]^{\Id} & H_1(Y(\eta);R) \oplus H_1(E_J;R) \ar[d]^{(\Id\oplus f_*)} \ar[r] & H_1(Y_J;R) \ar[r] \ar[d]^-{f_*} & 0 \\
H_1(T;R) \ar[r]^-{(\inc,0)}  & H_1(Y(\eta)) \oplus (H_1(E_U;R)=0) \ar[r] & H_1(Y;R) \ar[r] & 0.}\]
The lemma follows easily from this diagram and the above observations.
\end{proof}

The first part of Theorem \ref{thm:orthogonal-decomp-of-Bl-infected} states that there an isomorphism of $R$-modules
$\psi \colon H_1(Y;R) \oplus H_1(E_J;R) \toiso H_1(Y_J;R)$. The dashed map in the corollary below gives such an isomorphism. The corollary is immediate from Lemma~\ref{lemma:homology-of-Y}.  In Section~\ref{section:morphism-of-pairings}, we will prove that $\psi$ induces a morphism of linking pairings.

\begin{corollary}\label{cor:theIsometryMap}
Let $i_{Y_J} \colon Y(\eta) \rightarrow Y_J$, $i_Y \colon Y(\eta) \rightarrow Y$
and $i_J \colon E_J \rightarrow Y_J$ the inclusion maps. There is a unique morphism~$\psi$, depicted by a dashed arrow, making the following diagram commutative.
 \[\xymatrix@C-1.2cm@R1.0cm{
       & \ar[dl]_-{i_Y \oplus \Id}H_1(Y(\eta);R) \oplus H_1(E_J;R) \ar[dr]^-{i_{Y_J} + i_J }&  \\
           H_1(Y;R) \oplus H_1(E_J;R)  \ar@{-->}[rr]^-{\psi} &    &  H_1(Y_J;R).
}\]
Moreover the map $\psi$ is an isomorphism.
\end{corollary}


\subsection{Morphism of Blanchfield pairings}\label{section:morphism-of-pairings}


To prove Theorem \ref{thm:orthogonal-decomp-of-Bl-infected},
we have to show that the isomorphism $\psi$ is a morphism of linking pairings.

\begin{proposition}\label{prop:iso-lk-pairings}
The isomorphism $\psi \colon H_1(Y;R) \oplus H_1(E_J;R) \toiso H_1(Y_J;R)$ is an isomorphism of linking pairings.
\end{proposition}

\begin{proof}
We begin with a claim.
\begin{claim*}
  The following two pairings on $H_1(Y(\eta);R) \oplus H_1(E_J;R)$ agree:
\[ (i_Y \oplus \Id)^*(\Bl_{Y} \oplus \Bl_{E_J}) = (i_{Y_J} + i_J)^* \Bl_{Y_J}.\]
\end{claim*}

From the glueing formula of Theorem \ref{thm:decomp-of-3-mflds-Bl}, we
deduce that
\[(i_{Y_J} + i_J)^* \Bl_{Y_J} = \Bl_{Y(\eta)} \oplus \Bl_{E_J}\]
on $H_1(Y(\eta);R) \oplus H_1(E_J;R)$.
Express $Y$ as an infection by the unknot: $Y = Y_{U} = Y(\eta) \cup E_U$. Apply Theorem~\ref{thm:decomp-of-3-mflds-Bl} once again, to deduce that
$i_{Y}^* \Bl_{Y} = \Bl_{Y(\eta)}$, since $H_1(E_U;R)=0$.
Thus we obtain that
\[ (i_Y \oplus \Id)^*(\Bl_{Y} \oplus \Bl_{E_J}) = \Bl_{Y(\eta)} \oplus \Bl_{E_J}.\]
This completes the proof of the claim.
Next we show that $$(i_Y \oplus \Id)^*(\psi^*(\Bl_{Y_J})) = (i_Y \oplus \Id)^*(\Bl_Y \oplus \Bl_{E_J})$$
on $H_1(Y(\eta);R) \oplus H_1(E_J;R)$.
To see this we compute \begin{align*}
  (i_Y \oplus \Id)^*(\psi^*(\Bl_{Y_J})) &= (\psi \circ (i_Y \oplus \Id))^*(\Bl_{Y_J}) \\ &= (i_{Y_J} + i_J)^*(\Bl_{Y_J}) \\ &= (i_Y \oplus \Id)^*(\Bl_{Y} \oplus \Bl_{E_J}),\end{align*}
where the last equality is from the claim above.

Let $g:= i_Y \oplus \Id$.  Now we have two maps $\alpha := \psi^*(\Bl_{Y_J})$ and $\gamma := \Bl_Y \oplus \Bl_{E_J}$ on $H_1(Y;R) \oplus H_1(E_J;R)$ such that
\[g^{\wedge}\circ \alpha \circ g = g^{\wedge} \circ \gamma \circ g.\]
Note that $g= i_Y \oplus \Id$ is surjective, which implies that $g^{\wedge}$ is injective.  It follows that $\alpha=\gamma$, which completes the proof of Proposition~\ref{prop:iso-lk-pairings}.
\end{proof}

The proof of Theorem~\ref{thm:orthogonal-decomp-of-Bl-infected} is completed by our final  lemma.

\begin{lemma}\label{lemma:tensoring-up}
Let $\phi \colon \zt \to R$ be a monomorphism of Ore domains with involution and let $J$ be an oriented knot.
We have an isomorphism of linking pairings
\[\theta \colon (R \otimes_{\zt} H_1(E_J;\zt),R \otimes \Bl_{J}) \toiso (H_1(E_J;R),\Bl_{E_J}),\]
where $\theta$ is defined by 
$r\otimes [p\otimes \sigma]\to [r\phi(p)\otimes \sigma]$ for $\sigma\in C_1(\widetilde{E_J})$, $p\in \zt$ and $r\in R$.
\end{lemma}

This lemma was stated as Theorem~4.7 in \cite{Le06}, but at that time no proof was provided, as the result was not used in the rest of that paper. However, \cite[Theorem~4.7]{Le06} has since been cited by many subsequent papers, so we provide an argument here.

\begin{proof}
Consider the following diagram.
\[\xymatrix @C-0.8cm@R0.7cm{ H_1(E_J;R) \ar[dd]_-{\Bl_{E_J}} & R \otimes_{\Z[t^{\pm 1}]} H_1(E_J;\Z[t^{\pm 1}]) \ar[l]_-{\theta} \ar[d]^-{\Id \otimes \Bl_{J}}  \\
 & R \otimes_{\Z[t^{\pm 1}]} \ol{\Hom_{\Z[t^{\pm 1}]}(H_1(E_J;\Z[t^{\pm 1}]),\Q(t)/\Z[t^{\pm 1}])}   \ar[d]^-{\xi} \\
\ol{\Hom_R(H_1(E_J;R),Q/R)}\ar[r]^-{\theta^{\wedge}}   &  \ol{\Hom_R(R \otimes_{\Z[t^{\pm 1}]} H_1(E_J;\Z[t^{\pm 1}]),Q/R)} }  \]
The map $\theta$ an isomorphism of the underlying modules of the Blanchfield pairings.  The left vertical arrow is the Blanchfield pairing $\Bl_{E_J}$ of $E_J$ over $R$.  The composition of right vertical arrows  expresses the Blanchfield pairing $R \otimes \Bl_{J}$ defined in the introduction.   Recall that $\phi \colon \Z[t^{\pm 1}] \to R$ induces morphisms $\phi \colon \Q(t) \to Q$ and $\phi\colon \Q(t)/\Z[t^{\pm 1}] \to Q/R$. The map $\xi$ is defined as follows:
\[\xi \colon  r \otimes f    \mapsto  \Big( s\otimes d \mapsto  r \phi(f(d)) \ol{s} \Big).\]
The lemma follows from commutativity of the diagram.  To see that the diagram is commutative, note that every element in $H_1(E_J;R)$ can be written as (a sum of elements of the form) $[r\otimes d]$, where $r \in R$ and $d \in C_1(E_J;\zt)$, and that all the maps in the definition of the Blanchfield pairing are defined at the chain level.  The chain level maps are the same on the $C_*(E_J;\zt)$ (or $C^*(E_J;\zt)$ part, for the left and right vertical maps, and can always be taken to be the identity on the $R$ part.   For example, focussing on Poincar\'{e} duality, $PD(r \otimes d)$ is $$r \otimes PD(d) \in R \otimes C^2(E_J;\zt) \cong C^2(E_J;R).$$
Commutativity follows from a continuation of such definition chasing through the Bockstein and Kronecker maps.
\end{proof}

\end{document}